\documentclass[12pt,reqno]{article}
\usepackage{url}
\usepackage{fullpage}
\usepackage{amssymb}
\usepackage{amsmath}
\usepackage{amsthm}
\newtheorem{theorem}{Theorem}
\newtheorem{problem}{Problem}

\newtheorem{corollary}{Corollary}

\newtheorem{lemma}{Lemma}

\begin{document}

\title{\large\bf Weighted representation functions on $\mathbb{Z}_m$\footnote{This work was supported by the National
Natural Science Foundation of China, Grant No. 11071121 and  the
Project of Graduate Education Innovation of Jiangsu Province
(CXZZ12-0381). }}
\date{}
\author{Quan-Hui Yang and Yong-Gao Chen\footnote{Corresponding author. Email:yangquanhui01@163.com, ygchen@njnu.edu.cn}\\
\small School of Mathematical Sciences and Institute of
Mathematics,
\\ \small Nanjing Normal University,  Nanjing 210046, P. R. CHINA}
 \maketitle
\vskip 3mm

\begin{abstract}

Let $m$, $k_1$, and $k_2$ be three integers with $m\ge 2$. For any
set $A\subseteq \mathbb{Z}_m$ and $n\in \mathbb{Z}_m$, let
$\hat{r}_{k_1,k_2}(A,n)$ denote the number of solutions of the
equation $n=k_1a_1+k_2a_2$ with $a_1,a_2\in A$. In this paper,
using exponential sums, we characterize all $m$, $k_1$, $k_2$, and
$A$ for which
$\hat{r}_{k_1,k_2}(A,n)=\hat{r}_{k_1,k_2}(\mathbb{Z}_m\setminus
A,n)$ for all $n\in \mathbb{Z}_m$. We also pose several problems
for further research.



\end{abstract}

\section{Introduction}

Let $\mathbb{N}$ be the set of nonnegative integers. For a set
$A\subseteq \mathbb{N},$ let $R_1(A,n)$, $R_2(A,n),$ $R_3(A,n)$
denote the number of solutions of $a+a'=n, a,a'\in A$; $a+a'=n,
a,a'\in A,a<a'$ and $a+a'=n, a,a'\in A,a\leqslant a'$
respectively. For $i\in \{1,2,3\},$ S\'{a}rk\"{o}zy asked ever
whether there are sets $A$ and $B$ with infinite symmetric
difference such that $R_i(A,n)=R_i(B,n)$ for all sufficiently
large integers $n$. It is known that the answer is negative for
$i=1$ (see Dombi \cite{Dombi}) and the answer is positive for
$i=2,3$ (see Dombi \cite{Dombi}, Chen and Wang \cite{chen03}). In
fact, Dombi \cite{Dombi}
 for $i=2$ and Chen and Wang \cite{chen03} for $i=3$ proved that there exists a set $A\subseteq \mathbb{N}$
 such that $R_i(A,n)=R_i(\mathbb{N}\setminus A,n)$ for all $n\geq n_0.$ Lev (see \cite{lev}) gave a simple common proof to the
results by Dombi \cite{Dombi} and Chen and Wang \cite{chen03}.
Finally, using generating functions, S\'{a}ndor \cite{Sandor} gave
a complete answer by using generating functions, and later Tang
\cite{tang} gave an elementary proof. For related research, one
may refer to \cite{ChenSci2011} and \cite{ChenTangJNT}. For a
positive integer $m$, let $\mathbb{Z}_m$ be the set of residue
classes modulo $m.$ For the modular version, the first author and
Chen \cite{yang} proved that if and only if $m$ is even, there
exists $A\in \mathbb{Z}_m$ such that
$R_1(A,n)=R_1(\mathbb{Z}_m\setminus A,n)$ for all $n\in
\mathbb{Z}_m$.

For any given two positive integers $k_1,k_2$ and any set $A$ of
nonnegative integers, let $r_{k_1,k_2}(A,n)$ denote the number of
solutions of the equation $n=k_1a_1+k_2a_2$ with $a_1,a_2\in A$.
Recently, the authors \cite{chenyang} proved that there exists a
set $A\subseteq \mathbb{N}$ such that $r_{k_1, k_2}(A, n)=r_{k_1,
k_2}(\mathbb{N}\setminus A,n)$ for all sufficiently large integers
$n$ if and only if $k_1\mid k_2$ and $k_2>k_1$.

For any given $t$ integers $k_1,\cdots,k_t$, and any set
$A\subseteq \mathbb{Z}_m$ and $n\in \mathbb{Z}_m$, let
$\hat{r}_{k_1,\cdots,k_t}(A,n)$ denote the number of solutions of
the equation $n=k_1a_1+\cdots+k_ta_t$ with $a_1,\ldots,a_t\in A$.
In this paper, we prove the following theorem.

\begin{theorem}\label{mainthm1} Let $m,k_1$, and $k_2$ be three integers
with $m\geq 2$, and let $A\subseteq \mathbb{Z}_m$.  Then
$\hat{r}_{k_1,k_2}(A,n)=\hat{r}_{k_1,k_2}(\mathbb{Z}_m\setminus
A,n)$ for all $n\in \mathbb{Z}_m$ if and only if $|A|=m/2$ and $A$
is uniformly distributed modulo $d_1d_2/d_3^2$, where
$(k_1,m)=d_1,(k_2,m)=d_2$, and $(d_1,d_2)=d_3$.\end{theorem}

\begin{corollary}\label{cor1} Let $m,k_1$, and $k_2$ be three integers
with $m\geq 2$. Then there exists a set $A\subseteq \mathbb{Z}_m$
such that
$\hat{r}_{k_1,k_2}(A,n)=\hat{r}_{k_1,k_2}(\mathbb{Z}_m\setminus
A,n)$ for all $n\in \mathbb{Z}_m$ if and only if $m$ is even and
one of the following statements is true:

(i) $k_1$ and $k_2$ have the same parity;

(ii) $k_1$ and $k_2$ have the different parities with $v_2(k_i) <
v_2 (m) (i=1,2)$, where $v_2(k)=t$ if $2^t\mid k$ and
$2^{t+1}\nmid k$.
\end{corollary}

Motivated by Lev \cite{lev} and the authors \cite{chenyang}, we
now pose the following problems for further research.

\begin{problem}\label{prob1} For any given two integers $k_1$ and $k_2$,
determine all pairs of subsets $A,B\subseteq \mathbb{Z}_m$ such that
$\hat{r}_{k_1,k_2}(A,n)=\hat{r}_{k_1,k_2}(B,n)$ for all $n\in \mathbb{Z}_m$.
\end{problem}

\begin{problem}\label{prob2} For $t\geq 3$, find all $t+1$-tuples $(m, k_1, \dots, k_t)$
of integers for which there exists a set $A\subseteq \mathbb{Z}_m$
such that $\hat{r}_{k_1, \dots, k_t}(A, n)=\hat{r}_{k_1, \dots,
k_t}(\mathbb{Z}_m \setminus A, n)$ for all $n\in \mathbb{Z}_m $.
\end{problem}

\section{Proofs}

For $T\subseteq \mathbb{Z}_m$ and $x\in \mathbb{Z}_m$, let
 $$S_T(x)=\sum_{t\in T}e^{2\pi itx/m}.$$
Let $A\subseteq \mathbb{Z}_m$ and $B =\mathbb{Z}_m\setminus A$.
Then
\begin{equation*}\hat{r}_{k_1,k_2}(A,n)=\sum_{x=0}^{m-1}S_A(k_1x)S_A(k_2x)e^{-2\pi inx/m}
\end{equation*} for all $n\in \mathbb{Z}_m$. Let $g_A (x)=S_A(k_1x)S_A(k_2x)-S_B(k_1x)S_B(k_2x)$. Thus
\begin{equation}\label{basic1}\hat{r}_{k_1,k_2}(A,n)-\hat{r}_{k_1,k_2}(B,n)
=\sum_{x=0}^{m-1}g_A (x) e^{-2\pi inx/m}
\end{equation} for all $n\in \mathbb{Z}_m$.

In order to prove Theorem \ref{mainthm1}, we need the following
Lemmas.
\begin{lemma}\label{lemma1} Let $m,k_1,k_2$ be three  integers with $m\geq 2$.
If $\hat{r}_{k_1,k_2}(A,n)=\hat{r}_{k_1,k_2}(B,n)$ for all $n\in
\mathbb{Z}_m$, then $m$ is even and $|A|=m/2$.
\end{lemma}
\begin{proof} If $\hat{r}_{k_1,k_2}(A,n)=\hat{r}_{k_1,k_2}(B,n)$
holds for all $n\in \mathbb{Z}_m$, then we have
$$|A|^2=\sum_{n\in \mathbb{Z}_m}\hat{r}_{k_1,k_2}(A,n)
=\sum_{n\in \mathbb{Z}_m}\hat{r}_{k_1,k_2}(B,n)=|B|^2.$$ Hence we
get $|A|=|B|$, that is, $m$ is even and $|A|=m/2$.
\end{proof}

\begin{lemma}\label{lemma2} If $m\nmid k_ix \, (i=1,2)$, then $g_A(x)=0$.      \end{lemma}

\begin{proof} Since $m\nmid k_ix (i=1,2)$, it follows that
$$S_A(k_1x)+S_B(k_1x)=\sum_{j=0}^{m-1}e^{2\pi ik_1xj/m}=0$$ and $$S_A(k_2x)+S_B(k_2x)=\sum_{j=0}^{m-1}e^{2\pi ik_2xj/m}=0.$$
Hence $g_A(x)=S_A(k_1x)S_A(k_2x)-S_B(k_1x)S_B(k_2x)=0$.
\end{proof}

\begin{lemma}\label{lemma3} If $|A|= m/2$ and $m\mid k_ix \, (i=1,2)$, then $g_A(x)=0$.\end{lemma}

\begin{proof} Since $m\mid k_ix (i=1,2)$, it follows that
$$ S_A(k_1x)=|A|=S_A(k_2x) \quad \text{and} \quad S_B(k_1x)=|B|=S_B(k_2x).$$
Thus $g_A(x)=|A|^2-|B|^2$. By $|A|= m/2$ we have $|B|= m/2$.
Therefore, $g_A(x)=0$.\end{proof}

\begin{lemma}\label{lemma4} If $k$ and $\ell $ are two integers, then
\begin{eqnarray*} \sum_{\substack{x=0\\
m|kx}}^{m-1}S_T(\ell x)e^{-2\pi inx/m} = (k,m) \sum_{\substack{t\in
T\\ (k, m) |\ell t-n}}1.
\end{eqnarray*}\end{lemma}

\begin{proof}Let $d=(k,m)$.  Then \begin{eqnarray*} &&\sum_{\substack{x=0 \\ m|kx}}^{m-1}S_T(\ell x)e^{-2\pi inx/m}
  = \sum_{\substack{x=0 \\ m|kx}}^{m-1}\sum_{t\in T }e^{2\pi i(\ell t-n)x/m}\\
 &=&\sum_{s=0}^{d-1}\sum_{t\in T}e^{2\pi i(\ell t-n)s/d}
=d \sum_{\substack{t\in T \\ d|\ell t-n}}1.
\end{eqnarray*} \end{proof}

\begin{proof}[Proof of Theorem 1.]
By Lemma \ref{lemma1} we may assume that $m$ is even and $|A|=|B|=
m/2$. From \eqref{basic1}, by Lemmas \ref{lemma2}-\ref{lemma4}, we
have
\begin{eqnarray*}&&\hat{r}_{k_1,k_2}(A,n)-\hat{r}_{k_1,k_2}(B,n)\\
&=&\sum_{\substack{x=0 \\ m\nmid k_1x, m\nmid k_2x}}^{m-1}g_A (x)
e^{-2\pi inx/m}+\sum_{\substack{x=0 \\ m\mid k_1x}}^{m-1}g_A (x)
e^{-2\pi inx/m}\\
&& +\sum_{\substack{x=0 \\ m\mid k_2x}}^{m-1}g_A (x) e^{-2\pi
inx/m}-\sum_{\substack{x=0 \\ m\mid k_1x, m\mid k_2x}}^{m-1}g_A (x)
e^{-2\pi inx/m}\\
&=&\sum_{\substack{x=0 \\ m\mid k_1x}}^{m-1}g_A (x) e^{-2\pi
inx/m}+\sum_{\substack{x=0 \\ m\mid k_2x}}^{m-1}g_A (x) e^{-2\pi
inx/m}\\
&=&\frac m2 \sum_{\substack{x=0 \\ m\mid k_1x}}^{m-1}\left(
S_A(k_2x)-S_B(k_2x) \right) e^{-2\pi inx/m}\\
&& +\frac m2\sum_{\substack{x=0 \\
m\mid k_2x}}^{m-1}\left( S_A(k_1x)-S_B(k_1x) \right) e^{-2\pi
inx/m}\\
&=& \frac 12 md_1 \left( \sum_{\substack{a\in A \\ d_1|k_2
a-n}}1-\sum_{\substack{b\in B \\ d_1|k_2 b-n}}1 \right) +\frac 12
md_2 \left( \sum_{\substack{a\in A \\ d_2|k_1
a-n}}1-\sum_{\substack{b\in B \\ d_2|k_1 b-n}}1 \right) .
\end{eqnarray*}
It follows that
\begin{equation}\label{rkk}\hat{r}_{k_1,k_2}(A,n)=\hat{r}_{k_1,k_2}(B,n)\end{equation}
is equivalent to
\begin{equation}\label{c}d_1\sum_{\substack{a\in A \\ d_1|k_2a-n}}1+d_2\sum_{\substack{a\in A \\ d_2|k_1a-n}}1
=d_1\sum_{\substack{b\in B \\ d_1|k_2b-n}}1+d_2\sum_{\substack{b\in
B
\\ d_2|k_1b-n}}1.\end{equation}

Suppose that \eqref{rkk} holds for all $n\in \mathbb{Z}_m$. Then
\eqref{c} holds for all $n\in \mathbb{Z}_m$. Thus
\begin{equation}\label{c1}d_1\sum_{\substack{a\in A \\ d_1|k_2a-d_3n}}1+d_2\sum_{\substack{a\in A \\ d_2|k_1a-d_3n}}1
=d_1\sum_{\substack{b\in B \\
d_1|k_2b-d_3n}}1+d_2\sum_{\substack{b\in B
\\ d_2|k_1b-d_3n}}1\end{equation}for all $n\in \mathbb{Z}_m$.
Let
$$d_i=d_3 d_i',\quad k_i=d_3 k_i', \quad i=1,2.$$
From \eqref{c1}, we have
\begin{equation}\label{c2}d_1\sum_{\substack{a\in A \\ d_1'|k_2'a-n}}1+d_2\sum_{\substack{a\in A \\ d_2'|k_1'a-n}}1
=d_1\sum_{\substack{b\in B \\
d_1'|k_2'b-n}}1+d_2\sum_{\substack{b\in B
\\ d_2'|k_1'b-n}}1.\end{equation}
Since $(d_1, k_2)=(k_1, m, k_2)=d_3$, it follows that $(d_1',
k_2')=1$. Similarly, we have that $(d_2', k_1')=1$. Thus the
summation of two sides of \eqref{c2} is
$$ d_1\sum_{t\in \mathbb{Z}_m, d_1'|k_2't-n}1+d_2\sum_{t\in \mathbb{Z}_m, d_2'|k_1't-n}1
=d_1\sum_{t\in \mathbb{Z}_m,d_1'|t}1+d_2\sum_{t\in \mathbb{Z}_m,
d_2'|t}1=C(m, k_1, k_2)$$ (say). By \eqref{c2} we have
\begin{equation}\label{c3}d_1\sum_{\substack{a\in A \\ d_1'|k_2'a-n}}1+d_2\sum_{\substack{a\in A \\ d_2'|k_1'a-n}}1
=\frac 12 C(m, k_1, k_2)\end{equation} for all integers $n$. In
particular,
\begin{equation}\label{c4}d_1\sum_{\substack{a\in A \\ d_1'|k_2'a-d_1'n}}1+d_2\sum_{\substack{a\in A \\ d_2'|k_1'a-d_1'n}}1
=\frac 12 C(m, k_1, k_2)\end{equation} for all integers $n$. That
is,
\begin{equation}\label{c5}d_1\sum_{\substack{a\in A \\ d_1'|k_2'a}}1+d_2\sum_{\substack{a\in A \\ d_2'|k_1'a-d_1'n}}1
=\frac 12 C(m, k_1, k_2)\end{equation} for all integers $n$.
Thus \begin{equation}\label{c6}\sum_{\substack{a\in A \\
d_2'|k_1'a-d_1'n_1}}1=\sum_{\substack{a\in A \\
d_2'|k_1'a-d_1'n_2}}1\end{equation} for all integers $n_1$ and $
n_2$. Since $(d_1, d_2)=d_3$, we see that $(d_1', d_2')=1$.  By
\eqref{c6}, $(d_2', k_1')=1$, and $(d_1', d_2')=1$, we have
\begin{equation}\label{c7}\sum_{\substack{a\in A \\
d_2'|a-u_1}}1=\sum_{\substack{a\in A \\
d_2'|a-u_2}}1\end{equation} for all integers $u_1$ and $ u_2$. So
$A$ is uniformly distributed modulo $d_2'$. Similarly, $A$ is
uniformly distributed modulo $d_1'$. Since $(d_1', d_2')=1$, the set
$A$ is uniformly distributed modulo $d_1'd_2'=d_1d_2/d_3^2$.

Conversely, suppose that $A$ is uniformly distributed modulo
$d_1'd_2'=d_1d_2/d_3^2$. Then $A$ is uniformly distributed modulo
$d_1'$. So
\begin{equation}\label{c8}\sum_{\substack{a\in A \\
d_1'|a-n}}1=\frac{|A|}{d_1'}=\frac{md_3}{2d_1}\end{equation} for
all integers $n$. Since $(k_2', d_1')=1$, it follows that
\begin{equation*}\label{c9}\sum_{\substack{a\in A \\
d_1'|k_2' a-n}}1=\frac{md_3}{2d_1}\end{equation*} for all integers
$n$. That is,
\begin{equation}\label{c9}d_1\sum_{\substack{a\in A \\
d_1|k_2 a-d_3n}}1=\frac12 md_3\end{equation} for all integers $n$.
Similarly, we have
\begin{equation}\label{c10}d_2\sum_{\substack{a\in A \\
d_2|k_1 a-d_3n}}1=\frac12 md_3\end{equation} for all integers $n$.
Since $A$ is uniformly distributed modulo $d_1'd_2'=d_1d_2/d_3^2$,
the set $B=\mathbb{Z}_m\setminus A$ is also uniformly distributed
modulo $d_1'd_2'=d_1d_2/d_3^2$. Similarly, we have
\begin{equation}\label{c11}d_1\sum_{\substack{b\in B \\
d_1|k_2 b-d_3n}}1=\frac12 md_3\quad \text{and} \quad d_2\sum_{\substack{b\in B \\
d_2|k_1 b-d_3n}}1=\frac12 md_3\end{equation} for all integers $n$.
By \eqref{c9}, \eqref{c10}, and \eqref{c11}, we see that
\eqref{c1} holds for all integers $n$. That is, \eqref{c} holds
for all integers $n$ with $d_3|n$. For $d_3\nmid n$, \eqref{c}
holds trivially. So \eqref{c} holds for all $n\in \mathbb{Z}_m$.
Therefore, \eqref{rkk} holds for all $n\in \mathbb{Z}_m$.
\end{proof}

\begin{proof}[Proof of Corollary \ref{cor1}.] Suppose that there exists a set $A\subseteq \mathbb{Z}_m$
such that
$\hat{r}_{k_1,k_2}(A,n)=\hat{r}_{k_1,k_2}(\mathbb{Z}_m\setminus
A,n)$ for all $n\in \mathbb{Z}_m$. By Theorem \ref{mainthm1},
$|A|=m/2$ and $A$ is uniformly distributed modulo $d_1d_2/d_3^2$,
where $(k_1,m)=d_1,(k_2,m)=d_2$, and $(d_1,d_2)=d_3$. So $m$ is
even and $$\frac{d_1d_2}{d_3^2} \Big| \frac m2.$$  That is,
\begin{equation}\label{eqn1}\frac{2d_1d_2}{d_3^2} \Big| m.\end{equation} If $k_1$ and $k_2$ have the
different parities, say $k_1$ is even, then $v_2(d_1)=\min \{
v_2(k_1), v_2(m) \} $ and $v_2(d_2)=v_2(d_3)=0$. By \eqref{eqn1}
we have
$$1+v_2(d_1)=1+v_2(d_1)+v_2(d_2)-2 v_2(d_3)\le v_2(m).$$
So $v_2(k_1)=v_2(d_1)<v_2(m)$.

Conversely, suppose that $m$ is even and one of (i) and (ii) of
Corollary \ref{cor1} holds.

Since $d_1\mid m$, $d_2\mid m$, and $(d_1/d_3, d_2/d_3)=1$, it
follows that $d_1d_2/d_3^2 \mid m$.

If $k_1$ and $k_2$ are both odd, then $d_1/d_3$ and $d_2/d_3$ are
both odd. Noting that $m$ is even, we have that $2d_1d_2/d_3^2
\mid m$.

If $k_1$ and $k_2$ are both even, say $v_2(k_1)\ge v_2(k_2)$, then
$$v_2(d_1)=\min \{ v_2(k_1), v_2(m) \} \ge \min \{ v_2(k_2), v_2(m)
\} =v_2 (d_2).$$ So $v_2(d_3)=v_2(d_2)\ge 1$ and
$$v_2\left( \frac{2d_1d_2}{d_3^2}\right) =1+v_2(d_1)+v_2(d_2)-2
v_2(d_3)\le v_2(d_1)\le v_2(m).$$ Noting that $d_1d_2/d_3^2 \mid
m$, we have that $2d_1d_2/d_3^2 \mid m$.

If $k_1$ and $k_2$ have the different parities, say $k_1$ is even,
then $v_2(d_1)=v_2(k_1)< v_2(m)$ and $v_2(d_2)=v_2(d_3)=0$. Thus
$$v_2\left( \frac{2d_1d_2}{d_3^2}\right)
=1+v_2(d_1)+v_2(d_2)-2 v_2(d_3)=1+v_2(d_1)\le v_2(m).$$ By
$d_1d_2/d_3^2 \mid m$, we have that $2d_1d_2/d_3^2 \mid m$.

Write $d=d_1d_2/d_3^2$. Then $2d\mid m$. Let
$$ A= \bigcup_{i=1}^{d} \left\{ i+ d\ell : \ell =1,\dots , \frac{m}{2d} \right\}.$$
Then $|A|= m/2$ and $A$ is uniformly distributed modulo $d$. By
Theorem \ref{mainthm1},
$\hat{r}_{k_1,k_2}(A,n)=\hat{r}_{k_1,k_2}(\mathbb{Z}_m\setminus
A,n)$ for all $n\in \mathbb{Z}_m$.
\end{proof}


\end{document}